\documentclass[a4paper,11pt]{article}
\usepackage{latexsym}
\usepackage{amsmath,amsfonts,amssymb,amsthm}
\newtheorem{thm}{Theorem}[section]

\newtheorem{question}[thm]{Question}
\newtheorem{cor}[thm]{Corollary}
\newtheorem{prop}[thm]{Proposition}

\input epsf

\title{Universal convex coverings}
\author{Roland Bacher}

\begin{document}
\maketitle




\begin{abstract}
In every dimension $d\ge1$, we
  establish the existence of
a constant $v_d>0$ and of a
discrete subset $\mathcal U_d$ of
  $\mathbb R^d$ such that the following holds:
$\mathcal C+\mathcal U_d=\mathbb R^d$ for every convex set
  $\mathcal C\subset \mathbb R^d$ of volume at least $v_d$ 
and $\mathcal U_d$ contains
  at most $\log(r)^{d-1}r^d$ points at distance at most $r$ from the
  origin, for every large $r$.\footnote{ 
Math. class: 11H06, 52A21, 52C17}
\end{abstract}

\section{Introduction}

Fix a dimension $d\ge1$ and consider the volume associated to the
standard Lebesgue measure on $\mathbb R^d$. Given $v>0$, a
subset $\mathcal U$ of $\mathbb R^d$ is a {\it $v-$universal 
convex covering} of $\mathbb R^d$ if we have
$\mathcal C+\mathcal U=\mathbb R^d$
for every convex subset $\mathcal C$ of
$\mathbb R^d$ of volume strictly greater than $v$. 
Here, $\mathcal C+\mathcal U$ denotes the set of all points 
of the form $P+Q$ with $P$ in $\mathcal C$ and $Q$ in $\mathcal U$.

For every positive $t$, a subset $\mathcal U$ of $\mathbb R^d$
is a $v-$universal convex covering if and only if 
$t\mathcal U$ is a $(t^dv)-$universal convex 
covering. The properties in which we are interested are thus independent of
the particular value of $v$. We call $\mathcal U$ a {\it 
universal convex covering of
$\mathbb R^d$} if $\mathcal U$ is a $v-$universal convex covering 
of $\mathbb R^d$ for some $v>0$.

Our main result is the following.

\begin{thm} \label{thmmainA} Let $d\ge 1$. Up to translation and 
rescaling, any 
universal convex covering of the Euclidean vector space $\mathbb R^d$ 
has at least $\ell_d(r)=r^d$ points at distance at 
most $r$ from the origin. There
exists a universal convex covering $\mathcal U_d$ of $\mathbb R^d$ with at
most $u_d(r)=\log(r)^{d-1}r^d$ points at distance at most $r$ from the
origin, for every large $r$.
\end{thm}

The first part
of theorem~\ref{thmmainA} is obvious since, in any dimension $d$, one can consider the unit cube as the convex set
$\mathcal C$.
In dimension $d=1$, the second part is obvious too 
since $\mathcal I+\mathbb Z=\mathbb R$ for every
interval $\mathcal I$ of length strictly greater than $1$, hence one can choose
$\mathcal U_1=\mathbb Z$.
However, in dimension $d\ge2$, there is a factor of $\log^{d-1}$ 
between the easy lower bound $\ell_d$ and the upper bound $u_d$. 

Whilst it is surely possible to improve these results, I would be very 
surprised by the existence of universal coverings in dimension $d$ 
achieving the lower bound $\ell_d$ for $d\ge 2$.


Proposition \ref{propequiv} below suggests thus the following question.

\begin{question} \label{quest1} Let $\mathcal S$ be a discrete subset of the 
Euclidean plane $\mathbb R^2$ such that 
$$\sharp\{x\in S\ \vert\ \parallel x\parallel\leq R\}
\leq R^2+1$$ 
for all $R\geq 0$. Does the complement $\mathbb R^2
\setminus \mathcal S$ of $\mathcal S$ necessarily contain triangles
of arbitrarily large area?
\end{question}

By Proposition \ref{propequiv},
the anwser to question \ref{quest1} (which has an obvious generalization to
the case of dimension $d>2$) is YES if and only if there
is no universal covering of the plane achieving the lower 
bound $\ell_2$ for $d=2$.

Call a subset $\mathcal S$ of $\mathbb R^n$ {\it uniformly discrete} 
if there exists a neighbourhood $\mathcal O$ of the origin such that
$x-y\not\in \mathcal O$ for every pair $(x,y)$ of distinct elements in
$\mathcal S$.

As long as the answer to question \ref{quest1} is unknown or if the answer
is NO, the following question
and its higher-dimensional generalisations is also 
interesting.

\begin{question} Does the complement of a uniformly discrete subset 
$\mathcal S$ of the plane necessarily contain triangles of arbitrarily 
large area?
\end{question}


Universal convex coverings are related to
sphere coverings, see \cite{CS} for an overview, or more 
generally to coverings of $\mathbb R^d$ by translates of a fixed convex 
body. Rogers proved in \cite{R} that every convex body of $\mathbb R^d$
covers $\mathbb R^d$ with density at most $d(5+\log d+\log\log d)$
for a suitable covering. Erd\"os and Rogers in \cite{ER} showed the 
existence of such a covering which furthermore covers no point with
multiplicity exceeding $ed(5+\log d+\log\log d)$. 
Chapter 31 of \cite{Gr} contains an account of subsequent
developpements.

There appears to be no result in the literature closely related to
universal convex coverings and featuring results similar
to theorem~\ref{thmmainA}.

This paper is organized as follows.
In section~\ref{s.prelim} we collect some preliminary facts.
In section \ref{sectconstr} 
we construct recursively a sequence of sets
$(\mathcal U_d)_{d\ge1}$ such that $\mathcal U_1=\mathbb Z$ and $\mathcal 
U_d\subset \mathbb R^d$ for every $d\ge1$, and
we show, by induction on $d\ge1$, that $\mathcal U_d$ is a universal
convex covering of $\mathbb R^d$.
In section \ref{sectgrowthcl}, we define growth classes of functions and we introduce a natural equivalence
relation on them, which is compatible with the
natural partial order on increasing positive functions.
Finally, we show in section \ref{sectcompgrowth} that the growth class
of the universal convex covering $\mathcal U_d$
constructed in section \ref{sectconstr} is represented by 
$u_d$. This implies theorem \ref{thmmainA} by 
rescaling $\mathcal U_d$ suitably.

\section{Preliminaries}
\label{s.prelim}

For any subset $\mathcal S$ of $\mathbb R^d$, let
$$-\mathcal S=\{-Q\in\mathbb R^d\ \vert\
Q\in \mathcal S\}$$
denote the set of all opposite vectors.

\begin{prop} \label{propequiv}
Choose $v>0$. A subset $\mathcal U$ of $\mathbb R^d$
is a $v-$universal convex covering if and only if 
every convex subset of $\mathbb R^d$ with volume at least $v$ 
intersects $\mathcal U$ 
non-trivially.
\end{prop}

\begin{proof} 
Consider a convex subset $\mathcal C$ of $\mathbb R^d$ 
with volume at least $v$. Then $-\mathcal C$ is a  convex set of 
the same volume.
For any point $Q$ in $\mathbb R^d$, 
$Q$ belongs to $\mathcal C+\mathcal U$ if and only if the convex set
$-\mathcal C+Q$ intersects $\mathcal U$.
\end{proof}

The growth function $f_{\mathcal S}$ of a subset $\mathcal S\subset
\mathbb R^d$ without accumulation points is defined as 
follows: for $r$ an arbitrary positive real number,
$f_{\mathcal S}(r)$ denotes the number of points of 
$\mathcal S$ at distance at most $r$ from the origin.

Universal convex coverings
are stable under affine bijections and $v-$universal convex coverings
are stable under affine bijections which preserve the volume.
Thus we consider the growth class 
with respect to the equivalence relation $\sim$ defined as follows: 
for any  increasing non-negative functions $f$ and $g$, we have
$f\sim g$ if there exists a real number $t\ge 1$ such that
$f(r)\le g(tr)\le f(t^2r)$ for every $r\ge t$.

Growth functions of sets without accumulation points
related by affine bijections are
equivalent under this equivalence relation.

For any nonzero integer $n$, let $v_2(n)$ denote the $2$-valuation 
of $n$: this is the unique integer $k$ such that $n$ is $2^k$ times 
an odd integer. 
Write any point $x$ of $\mathbb R^d$ as $x=(x_i)_{1\le i\le d}$, use 
the coordinate functions $\pi_i$ defined by $\pi_i(x)=x_i$ and let 
$\rho^{(i)}$ denote the projection of $\mathbb R^d$ onto 
$\mathbb R^{d-1}$ obtained by erasing the $i-$th coordinate $x_i$
and defined by
$$\rho^{(i)}(x_1,\dots,x_d)=(x_1,\dots,x_{i-1},\hat{x_i},
x_{i+1},\dots,x_d)\ .$$

\section{From dimension $d$ to dimension $d+1$}
\label{sectconstr}

Let $\mathcal U$ denote a subset of $\mathbb R^d$.
For every $1\le i\le d+1$, let $\varphi_i^{d}(\mathcal U)$ denote the 
set of points $ x=(x_j)_{1\le j\le d+1}$ in $\mathbb R^{d+1}$
such that $x_i\in\mathbb Z\setminus\{0\}$ and
$2^{v_2(x_i)/d}\rho^{(i)}(x)$ belongs to $\mathcal U$.
Finally, let
$$
\varphi_d(U)=\bigcup_{i=1}^{d+1}\varphi_i^{d}(U).
$$
For example,
$\varphi_1(\mathbb Z)\subset\mathbb R^2$ is the set of 
all points $(x,y)\in(\mathbb Z[\frac{1}{2}])^2$ such that 
$xy\in\mathbb Z\setminus\{0\}$ or $xy=0$ and 
$x+y\in\mathbb Z\setminus \{0\}$. Otherwise stated,
a point $(x,y)$ of $\varphi_1(\mathbb Z)$ is either
a non-zero element of $\mathbb Z^2$ or it has two non-zero coordinates 
and belongs to the set
$$\bigcup_{n=0}^\infty \big((2^n\mathbb Z)\times (2^{-n}\mathbb Z)
\big)\cup
\big((2^{-n}\mathbb Z)\times (2^n\mathbb Z)\big)\ .$$

\begin{prop} \label{propunivconvconstr}
Let $\mathcal U$ be a $v-$universal convex covering of
$\mathbb R^d$. Then $\varphi_d(\mathcal U)$ is a 
$v'-$universal covering of $\mathbb R^{d+1}$ with $v'$ given by
$$
v'=4^{d+1}\max(1,4v).
$$
\end{prop}

The value of $v'$ in proposition \ref{propunivconvconstr} is not optimal and
can easily be improved.

Let $(\mathcal U_d)_{d\ge 1}$ denote the sequence of sets defined recursively
by $\mathcal U_1=\mathbb Z$ and, for every $d\ge1$,
$$
\mathcal U_{d+1}=\varphi_d(\mathcal U_d).
$$
Proposition \ref{propunivconvconstr} implies the following result.

\begin{cor}\label{corconstrSn}
For every $d\ge1$, the set $\mathcal U_d$ is a universal convex covering 
of $\mathbb R^d$.
\end{cor}

\begin{proof}[Proof of proposition \ref{propunivconvconstr}]
  By proposition \ref{propequiv}, it is enough to show that the volume
  of any convex set $\mathcal C$ which does not intersect 
$\varphi_d(\mathcal U)$ is
  bounded by $v'$.  Without loss of generality, 
we may assume that $\mathcal C$ is
  open.  Let $L$ denote the diameter of $\mathcal C$ with respect 
to the $L^\infty$ norm on $\mathbb R^{d+1}$ defined by $\parallel x
\parallel_\infty=\max_{1\leq i\leq d+1}(\vert x_i\vert)$.
  Two cases arise.

First, if $L\le 4$, the volume of $\mathcal C$ satisfies 
$\mathop{vol}(\mathcal C)\le 4^{d+1}\le v'$.

The remaining case is when $L>4$. Hence we assume that $L>4$ and we
must show that the volume $\mathop{vol}(\mathcal C)$ of 
$\mathcal C$ is at most $4^{d+2}v$.

Note that there exists an index $1\le i\le d+1$ such that
$\pi_i(\mathcal C)=]a,b[\subset \mathbb R$ is an open interval of length $L$.
Thus one can pick two real numbers $\alpha$ and $\beta$ such that
$$a<a+\frac{L}{4}\leq \alpha<\alpha+\frac{L}{4}\leq \beta<\beta+\frac{L}{4}
\leq b$$
and $\alpha\beta\ge 0$ (or, equivalently, $\alpha$ and $\beta$
are of the same sign).

Then the interval $]\alpha,\beta[$ contains 
an integer $k$ such that $\vert k\vert =2^m\ge L/8$.
This implies that ${\mathcal C}'=\pi_i^{-1}(\{k\})$ is a convex set
of $\mathbb R^d$ which does not intersect $2^{-m/d}\mathcal U$.
By proposition \ref{propequiv}, the volume $\mathop{vol}({\mathcal C}')$ 
of ${\mathcal C}'$ is at most $v/2^m\le 8v/L.$

Let ${\mathcal C}_-$ denote the set of points $x$ in 
$\mathcal C$ such that $x_i\le k$, and
let ${\mathcal C}_+$ denote the set of points $x$ in 
$\mathcal C$ such that $x_i\ge k$.
Then,
$$\mathop{vol}({\mathcal C}_-)\le (k-a)\,\mathop{vol}({\mathcal C}')
\left(\frac{b-a}{b-k}\right)^d\leq L\,\frac{8v}{L}\,
\left(\frac{L}{L/4}\right)^d
\le 2\cdot 4^{d+1}v.
$$
The same inequality holds for $\mathop{vol}({\mathcal C}_+)$.
Since $\mathop{vol}(\mathcal C)=\mathop{vol}({\mathcal C}_+)+
\mathop{vol}({\mathcal C}_-)$, 
we have $\mathop{vol}(\mathcal C)\le 4^{d+2}v
\le v'$.
\end{proof}

\section{Growth classes}
\label{sectgrowthcl}

Let $\mathcal G_0$ denote the set of positive and increasing functions $f$
defined on an interval $[M(f),+\infty[$, where $M(f)$ is a finite real
number which may depend on $f$. Then $\mathcal G_0$ is equipped 
with a preorder
relation $\preceq$ defined by $f\preceq g$ if there exists $t\ge 1$
such that $f(x)\le g(tx)$ for every $x\ge t$.

The set $\mathcal G$ of {\it (affine) growth classes} is the quotient set 
of $\mathcal G_0$ by the equivalence relation $\sim$ defined by 
$f\sim g$ if there exists $t\ge 1$ such  that, for every $x\ge t$,
$$
g(x)\le f(tx)\le g(t^2 x).
$$
The preorder relation $\preceq $ on 
$\mathcal G_0$ induces a partial order on $\mathcal G$.

Recall that for every $a>0$ and $x>0$, $\ell_a(x)=x^a$, hence
 $\ell_a\in G_0$.
A function $f\in\mathcal G$ is 
{\it polynomially bounded} if there exists $a>0$ such that $f\preceq \ell_a$. 
If $f$ is polynomially bounded, $f$ has {\it critical exponent} $a>0$
if $\ell_{b}\preceq f\preceq \ell_{c}$ for every positive $b$ and $c$ 
such that $b<a<c$. Additionally,
a non-zero function
$f$ has {\it critical exponent} $0$ if $f\preceq \ell_{b}$ for every
$b>0$.

Equivalently, a function $f\in\mathcal G$ is polynomially bounded
if $\limsup_{x\rightarrow\infty}\frac{\log(f(x))}{\log(x)}<\infty$
and we have $a=\lim_{x\rightarrow\infty}
\frac{\log(f(x))}{\log(x)}$ if $f\in\mathcal G$ is polynomially
bounded with critical exponent $a$.

It can happen that a polynomially bounded function has no critical 
exponent. This is the case if
$\sup\{a\ \vert\ \ell_a\preceq f\}<\inf\{a\ \vert \ f\preceq
\ell_a\}$.  

Any function $f$ with critical exponent $a$ can be written as
$f=\ell_a\,h$, where the (not necessarily eventually increasing) function
$h$ is such that, for every $b>0$, there exists a finite $x_b$
such that $x^{-b}\le h(x)\le x^b$ for every $x\ge x_b$.

The notions of polynomial boundedness and critical exponent of
functions in $\mathcal G_0$ are well behaved with respect to the preorder
relation $\preceq$ on $\mathcal G_0$, hence these can also be defined on
suitable growth classes in $\mathcal G$.

Given $\mathcal S\subset \mathbb R^d$ without accumulation points, 
the choice of a (not
necessarily Euclidean) norm on $\mathbb R^d$
yields an increasing non-negative function 
$f_{\mathcal S}$ such that $f_{\mathcal S}(r)$ is the number of elements
of $\mathcal S$ whose norm is at most $r$.
The growth class of $f_{\mathcal S}$ is independent of the norm, 
hence one can call it the growth class of $\mathcal S$.
Two subsets of $\mathbb R^d$
related by a translation belong to the same growth class.
Growth classes are invariant
 under the action of the group of affine bijections of
$\mathbb R^d$.

A set $\mathcal S\subset \mathbb R^d$ 
is {\it sparse}
if its growth class is strictly smaller than $\ell_d$.
We say that $\mathcal S\subset \mathbb R^d$ is {\it nearly uniform}
if it has a polynomially bounded growth class of critical exponent $d$. 
The growth class  of a nearly uniform 
set can be represented by a function $h\ell_d$, where $h$
encodes  
the ``asymptotic density'' of $\mathcal S$ up to affine bijections.

For example,
$\mathbb Z^d$ and $\mathbb N^d$ are both nearly uniform sets
and in the same growth class $\ell_d$.

A more concise and less precise reformulation of theorem \ref{thmmainA} 
is as follows.

\begin{thm} In every dimension, there exist nearly uniform universal 
convex coverings.
\end{thm}

We conclude this section with a remark.

 One can also define growth classes for measurable
  subsets $\mathcal A\subset \mathbb R^d$ and for any given measure $\mu$ and 
norm on $\mathbb R^d$, by replacing $f_{\mathcal A}$ by the function
  $f_{\mathcal A}^\mu$ such that $f_{\mathcal A}^\mu(r)$ denotes 
the $\mu$ measure of
  the intersection of $\mathcal A$ with the ball of radius $r$ centered at the
  origin.

\section{Growth class of $\mathcal U_d$}
\label{sectcompgrowth}

Recall that $u_d(r)=\log(r)^{d-1}r^d$ for every $r\ge1$.

\begin{prop} \label{propgrcl}
The universal convex covering $\mathcal U_d$ defined in 
corollary \ref{corconstrSn}
belongs to the growth class of $u_d$.
\end{prop}

\begin{proof}[Proof of theorem \ref{thmmainA}] 
By proposition \ref{propgrcl}, there exists a constant $c_d$  such that
the set $\mathcal U_d$ constructed in corollary \ref{corconstrSn} 
has at most $c_du_d(r)$ elements at distance
at most $r$ from the origin. Considering the rescaled set
$t\mathcal U_d$ for $t>c_d^{1/d}$ ends the proof.
\end{proof}

\begin{proof}[Proof of proposition \ref{propgrcl}] We proceed by 
induction on the dimension $d$. For $d=1$,
$u_1(r)=r$ hence $\mathcal U_1=\mathbb Z$ belongs to the growth class
of $u_1$.

Before starting the proof of the induction step, let us remark
that the growth class of the function $u_d$ contains 
all functions in $\mathcal G$ which can be written as
$\lambda(r)u_d(r)$
where $r\longmapsto \lambda(r)$ is a bounded function.
This fact allows to neglect bounded factors involved in $u_d$
or $u_{d+1}$.

We assume now that $\mathcal U_d$ is 
of growth class $u_d$ for some $d\ge 1$. 
Up to a bounded factor,
the growth class of $\mathcal U_{d+1}$ is described by the set 
$\mathcal B\subset\mathbb N\times \mathbb R^d$ defined as
$$
\mathcal B=
\bigcup_{m\ge 1}(m,2^{-v_2(m)/d}{\mathcal U}_d)
=
\bigcup_{n\ge 0}(2^n(1+2\mathbb N),2^{-n/d}{\mathcal U}_d).
$$

We work with the $L^\infty$ norm $\parallel x\parallel_\infty$
already encountered in section \ref{sectconstr}.
Using the fact that growth classes are increasing and 
that bounded factors in $u_{d+1}$ can be neglected, it
is enough to compute the growth function 
$$\beta(r)=\sharp\left(\mathcal B\cap\{x\in\mathbb R^{d+1}\ \vert 
\ \parallel x\parallel_\infty<r\}\right)$$
counting elements of $\mathcal B$ in open balls of radius a power of $2$.

Neglecting boundary effects and using the fact that 
the set $\{1,\dots,2^m-1\}$ contains exactly 
$2^{m-n-1}$ integers of the form $2^n(1+2\mathbb N)$,
we have 
$$\beta(2^m)\sim\sum_{n=0}^{m-1}2^{m-n-1}u_d(2^{m+n/d})$$
$$\sim\sum_{n=0}^{m-1}2^{m-n-1}\left(m+\frac{n}{d}\right)^{d-1}2^{dm+n}
\sim m^d2^{(d+1)m}$$
which shows that $\beta$ is in the growth class of $u_{d+1}$.
\end{proof}

{\bf Acknowledgements} I thank D. Piau, P. de la Harpe, 
G. Mac Shane, B. Sevennec 
and Y. Colin de Verdi\`ere for interesting discussions and 
for many helpful remarks improving the exposition.


\noindent Roland BACHER

\noindent INSTITUT FOURIER

\noindent Laboratoire de Math\'ematiques

\noindent UMR 5582 (UJF-CNRS)

\noindent BP 74

\noindent 38402 St Martin d'H\`eres Cedex (France)
\medskip

\noindent e-mail: Roland.Bacher@ujf-grenoble.fr
\end{document}